\documentclass[a4paper, russian, 12pt]{article}

\usepackage[OT1]{fontenc}                         
\usepackage[utf8]{inputenc}                       
\usepackage[russian]{babel}

\usepackage{cmap} 

\usepackage{amsopn,amsmath,amssymb,amsthm}
\usepackage[dvipsnames]{xcolor}
\usepackage{graphicx}
\usepackage{subcaption}
\usepackage{mathtools}
\usepackage{enumerate}

\usepackage{csquotes}

\usepackage[pdfencoding=unicode, psdextra, citecolor=blue, urlcolor=blue,
  linkcolor=blue, colorlinks=true, bookmarksopen=true]{hyperref}

\usepackage{indentfirst}
\usepackage[top=1.5cm,bottom=1.5cm,left=1.5cm,right=1.5cm]{geometry}
\usepackage{comment, cite}

\usepackage{graphicx}
\graphicspath{{pics/}}

\renewcommand\geq{\geqslant}

\newcommand{\R}{\mathbb{R}}
\newcommand{\Z}{\mathbb{Z}}

\newcommand{\pder}[2]{\frac{\partial \, #1}{\partial \, #2} }

\theoremstyle{definition}
\newtheorem{definition}{Определение}

\theoremstyle{remark}

\theoremstyle{plain}
\newtheorem{theorem}{Теорема}
\newtheorem{lemma}{Лемма}
\newtheorem{proposition}{Предложение}
\newtheorem{corollary}{Следствие}

\theoremstyle{definition}

\newtheorem{example}{Пример}

\def\g{\mathfrak{g}}
\def\se{\mathfrak{se}}
 
\def\ga{\gamma}
\def\f{\varphi}
\def\lam{\lambda}

\def\spann{\operatorname{span}\nolimits}

\def\SE{\operatorname{SE}\nolimits}
\def\SO{\operatorname{SO}\nolimits}
\def\SU{\operatorname{SU}\nolimits}
\def\SL{\operatorname{SL}\nolimits}
\def\SH{\operatorname{SH}\nolimits}
\def\Id{\operatorname{Id}\nolimits}
\def\Isom{\operatorname{Isom}}
\def\ad{\operatorname{ad}\nolimits}
\def\cut{\operatorname{cut}\nolimits}
\def\tcut{t_{\cut}}
\def\tu{\widetilde{u}}
\def\tg{\widetilde{g}}
\def\og{\overline{g}}

\renewcommand\Vec{\operatorname{Vec}\nolimits}

\newcommand{\be}[1]{\begin{equation}\label{#1}}
\newcommand{\ee}{\end{equation}}
\newcommand{\eq}[1]{$(\protect\ref{#1})$}
\newcommand{\map}[3]{#1 \, : \, #2 \to #3}
\newcommand{\mapto}[3]{#1 \, : \, #2 \mapsto #3}

\graphicspath{{figures/}}

\begin{document}



\begin{center}
\Large
Однородные субримановы геодезические\\на группе движений плоскости\footnote{Исследование  
выполнено за счет гранта Российского научного фонда 
(проект № 17-11-01387-П) в Институте программных систем им. А.К. Айламазяна Российской академии наук.}
\\[1cm]

Ю.\,Л. Сачков

\today

\begin{abstract}
Описаны однородные субримановы геодезические для стандартной субримановой структуры на группе собственных движений плоскости $\SE(2)$. Показано, что эта структура не является геодезически орбитальной, несмотря на инвариантность времени разреза при сдвиге начальной точки вдоль геодезических.
\end{abstract}
\end{center}

В римановой геометрии известны понятия однородных геодезических и геодезически орбитальных пространств \cite{KV, KS}. В субримановой геометрии они практически не исследованы, нам известна на эту тему только работа \cite{podobryaev_GO}. Цель данной заметки --- изучение этих свойств для стандартной субримановой структуры на группе собственных движений плоскости, включая их связь с инвариантностью времени разреза при сдвиге начальной точки вдоль геодезических.

\section{Однородные и эквиоптимальные субримановы \\геодезические на группах Ли}
Пусть на гладком многообразии $M$ задана   субриманова структура \cite{montgomery, ABB}. Обозначим через $\Isom(M)$ группу изометрий субриманова многообразия $M$. 

\begin{definition}
Субриманова геодезическая $\ga \subset M$  называется {\em однородной}, если она является однородным пространством некоторой однопараметрической подгруппы в $\Isom(M)$, т.е. существует однопараметрическая подгруппа $\{\f_s \mid s \in \R\} \subset \Isom(M)$  такая, что:
\begin{enumerate}
\item
$\forall \ s \in \R \quad \f_s(\ga) \subset \ga$, 
\item
$\forall \ g_1, \, g_2 \in \ga \quad \exists \ s \in \R \ : \quad \f_s(g_1) = g_2$.
\end{enumerate}

Субриманово многообразие называется {\em геодезически орбитальным}, если все его геодезические однородны.
\end{definition}

\begin{definition}
{\em Временем разреза} для геодезической $g(t)$, $t \geq 0$, соответствующим начальному моменту $t = 0$, называется величина
$$
\tcut(g(\cdot)) = \sup \{ T > 0 \mid g(t) \text{ оптимальна при } t \in [0, T]\}.
$$

Геодезическая
$$
g(t), \quad t \in [0, T], \qquad T = \tcut(g(\cdot)),
$$
называется {\em непродолжаемой кратчайшей}.
\end{definition}

Пусть на группе Ли $G$ задана левоинвариантная субриманова структура с ортонормированным репером $(X_1, \dots, X_k)$, $X_i \in \Vec(G)$. 
Будем говорить, что геодезическая $\{g(t)\} \subset G$  соответствует управлению $u(t) = (u_1(t), \dots, u_k(t))$, $u_i \in L^{\infty}$, если $\dot g(t) = \sum_{i=1}^k u_i(t) X_i(g(t))$.

Согласно принципу максимума Понтрягина \cite{PBGM, notes}, любая нормальная геодезическая $\{g(t)\} \subset  G$  есть проекция нормальной экстремали $\{\lam(t)\} \subset T^*G$:
\be{Ham}
\dot\lam(t) = \vec{H}(\lam(t)), \qquad \lam(t) \in T_{g(t)}^* G, 
\ee
где $\vec{H} \in \Vec(T^*G)$  есть гамильтоново векторное поле с гамильтонианом $H = \frac 12 \sum_{i=1}^k h_i^2 \in C^{\infty}(T^*G)$, $h_i(\lam) = \langle \lam, X_i(\pi(\lam))\rangle$, и $\map{\pi}{T^*G}{G}$  есть каноническая проекция.

Кокасательное расслоение $T^*G$  группы Ли $G$  тривиализуется левыми сдвигами $\mapto{L_g}{g_0}{gg_0}$, $g, g_0 \in G$:
\begin{align*}
&\map{\Phi}{\g^* \times G}{T^*G}, \qquad (p, g) \mapsto L_g^*p,\\
&\langle L_g^*p, L_{g*} \xi\rangle = \langle p, \xi\rangle, \qquad p \in \g^*, \ \xi \in \g, \ g \in G.
\end{align*}
В этой тривиализации гамильтонова система \eq{Ham}  становится треугольной, см. \cite{notes}:
\begin{align}
&\dot p = \left( \ad \frac{\partial  H}{\partial  p} \right)^* p, \qquad p \in \g^*, \label{vert}\\
&\dot g = L_{g_*} \frac{\partial H}{\partial  p}, \qquad g \in G. \nonumber
\end{align}
Обозначим вертикальную компоненту гамильтонова поля в правой части уравнения \eq{vert}  через $\vec{H}_v \in \Vec(\g^*)$.

Переходя к натурально параметризованным геодезическим ($\sum_{i=1}^k u_i^2(t) \equiv 1$),  будем считать, что $p \in C := \g^* \cap \{ H = \frac 12 \}$. Тогда время разреза на нормальных геодезических $g(t) = \pi \circ e^{t \vec H}(p, \Id)$  становится функцией $\map{\tcut}{C}{(0, + \infty]}$.

\begin{definition}
Пусть $g(t)$, $t \in \R$,  есть натурально параметризованная геодезическая в субримановом многообразии $M$.  Геодезическая $g(t)$  называется {\em эквиоптимальной}, если она удовлетворяет следующему свойству: 
 если $g(t)$, $t \in [0, T]$,  есть непродолжаемая кратчайшая, то для любого $\tau \in \R$  геодезическая $g(t+\tau)$, $t \in [0, T]$,  есть также непродолжаемая кратчайшая.

Субриманово многообразие $M$ называется {\em эквиоптимальным}, если любая его натурально параметризованная геодезическая эквиоптимальна.
\end{definition}

\begin{lemma}\label{lem:t1}
Пусть $\{g(t)\} \subset  G$  есть субриманова геодезическая с управлением $u(t)$, и пусть $g_1 \in G$. Тогда кривая $\tg(t) = g_1 g(t+\tau)$  есть субриманова геодезическая с управлением $\tu(t) = u(t+\tau)$.
\end{lemma}
\begin{proof}
Кривая $g(t+\tau)$  есть геодезическая по определению геодезической, а ее левый сдвиг $\tg(t) = L_{g_1}(g(t+\tau))$ есть геодезическая в силу левоинвариантности субримановой структуры. Вычислим управление $\tu(t)$, соответствующее геодезической $\tg(t)$:
\begin{align*}
\dot{\tg}(t) &= \frac{d}{dt} L_{g_1}(g(t+\tau)) = L_{g_1*} \dot g(t+\tau) = 
L_{g_1*} \sum_{i=1}^k u_i(t+\tau) X_i(g(t+\tau)) = \\
&=  \sum_{i=1}^k u_i(t+\tau) L_{g_1*} X_i(g(t+\tau)) = \sum_{i=1}^k u_i(t+\tau)   X_i(\tg(t)),
\end{align*}
поэтому $\tu(t) = u(t+\tau)$.
\end{proof}

\begin{proposition}\label{prop:tcut_inv}
Нормальная геодезическая $g(t) = \pi \circ e^{t \vec H}(p, \Id)$, $t \in \R$,  эквиоптимальна тогда и только тогда, когда
время разреза инвариантно относительно выбора начального момента, т.е.
\be{tcut}
\tcut \circ e^{\tau \vec H_v}(p) = \tcut(p),
\qquad  p \in C, \quad \tau \in \R.
\ee
\end{proposition}

\begin{proof}
Геодезическая $\og(t) = g(t+\tau)$, $t \in [0, T]$,  есть   непродолжаемая кратчайшая тогда и только тогда, когда таковой является геодезическая $\tg(t) = g_1^{-1} \og(t) = g_1^{-1} g(t+\tau)$, $g_1 = g(\tau)$, $t \in [0, T]$. Если геодезическая $g(t)$  соответствует управлению $u(t)$, то геодезическая $\tg(t)$  соответствует управлению $\tu(t) = u(t+\tau)$,  см. лемму \ref{lem:t1}. Наконец, равенство \eq{tcut} означает, что для любого $T > 0$ управления $u(t)$, $t \in [0, T]$,  и  $u(t+\tau)$, $t \in [0, T]$, одновременно оптимальны или неоптимальны.
\end{proof}

\begin{corollary}\label{cor:GO}
Стандартные левоинвариантные субримановы структуры на следующих группах Ли  эквиоптимальны:
\begin{enumerate}
\item
группа Гейзенберга,
\item
группы $\SO(3)$, $\SU(2)$, $\SL(2)$  с осесимметричной субримановой метрикой,
\item
группы $\SE(2)$ и $\SH(2)$,
\item
группы Энгеля и Картана.
\end{enumerate}
\end{corollary}
\begin{proof}
Все геодезические для указанных групп Ли нормальны.
Инвариантное свойство времени разреза \eq{tcut} для этих групп Ли доказано в соответствующих статьях:
\begin{enumerate}
\item
\cite{versh_gersh},
\item
\cite{boscain_rossi},
\item
\cite{cut_sre2, sh2_3},
\item
\cite{engel_cut, cartan_cut}.
\end{enumerate}
\end{proof}

\begin{proposition}\label{prop:equi}
Если  субриманова геодезическая однородна, то она эквиоптимальна.
\end{proposition}
\begin{proof}
Пусть геодезическая $\ga = \{g(t) \mid t \in \R\}$, однородна, и пусть $\{\f_s \mid s \in \R\}$ есть однопараметрическая подгруппа в $\Isom(G)$, для которой $\ga$  есть однородное пространство. Далее, пусть $g(t)$, $t \in [0, T]$  есть непродолжаемая кратчайшая. Возьмем любое $\tau \in \R$ и найдем такое $s \in \R$, что $\f_s(g(0)) = g(\tau)$. Тогда $g(t+\tau)$, $t \in [0, T]$,  есть непродолжаемая кратчайшая так как $g(t+\tau) = \f_s(g(t))$, $t \in \R$. 
\end{proof}

В следующем разделе мы покажем ложность импликации, обратной к  предложению \ref{prop:equi}, на примере стандартной субримановой структуры на группе $\SE(2)$  {\cite{max_sre, cut_sre1, cut_sre2}}.

\section{Субримановы геодезические на группе $\SE(2)$}
Группа собственных движений евклидовой плоскости $\SE(2)$  есть полупрямое произведение группы параллельных переносов $\R^2$  и группы вращений плоскости $\SO(2)$: $\SE(2) = \R^2 \ltimes \SO(2)$. Эта группа имеет линейное представление
$$
\SE(2) = \left\{\left(
\begin{array}{ccc}
\cos \theta & - \sin \theta & x \\
\sin \theta & \cos \theta & y \\
0 & 0 & 1 
\end{array}
\right) \mid \theta \in S^1, \ x, \, y \in \R \right\}.
$$
Наряду с матричным обозначением, будем обозначать элементы этой группы как $(x,y,\theta)$.

Рассмотрим на группе $\SE(2)$  стандартную левоинвариантную субриманову структуру, порожденную ортонормированным репером
$$
X_1 = \cos \theta \pder{}{x} + \sin \theta \pder{}{y}, \qquad X_2 = \pder{}{\theta}.
$$
Субримановы геодезические и оптимальный синтез для этой структуры построены в работах \cite{max_sre, cut_sre1, cut_sre2}.

\begin{example}\label{ex:se2}
Отметим следующие геодезические для этой структуры:
\begin{itemize}
\item[$(1)$]
<<движение вперед>> $(x,y,\theta) = (t, 0, 0)$, $t \in \R$,
\item[$(2)$]
<<поворот на месте>> $(x,y,\theta) = (0, 0, t)$, $t \in \R$.
\end{itemize}

Проекции всех остальных геодезических на плоскость $(x,y)$  являются некомпактными кусочно-гладкими кривыми с точками возврата, см. Рис. \ref{fig:xyC1}--\ref{fig:xyC3}.
\end{example}

\begin{figure}[htbp]
\includegraphics[width=0.32\textwidth]{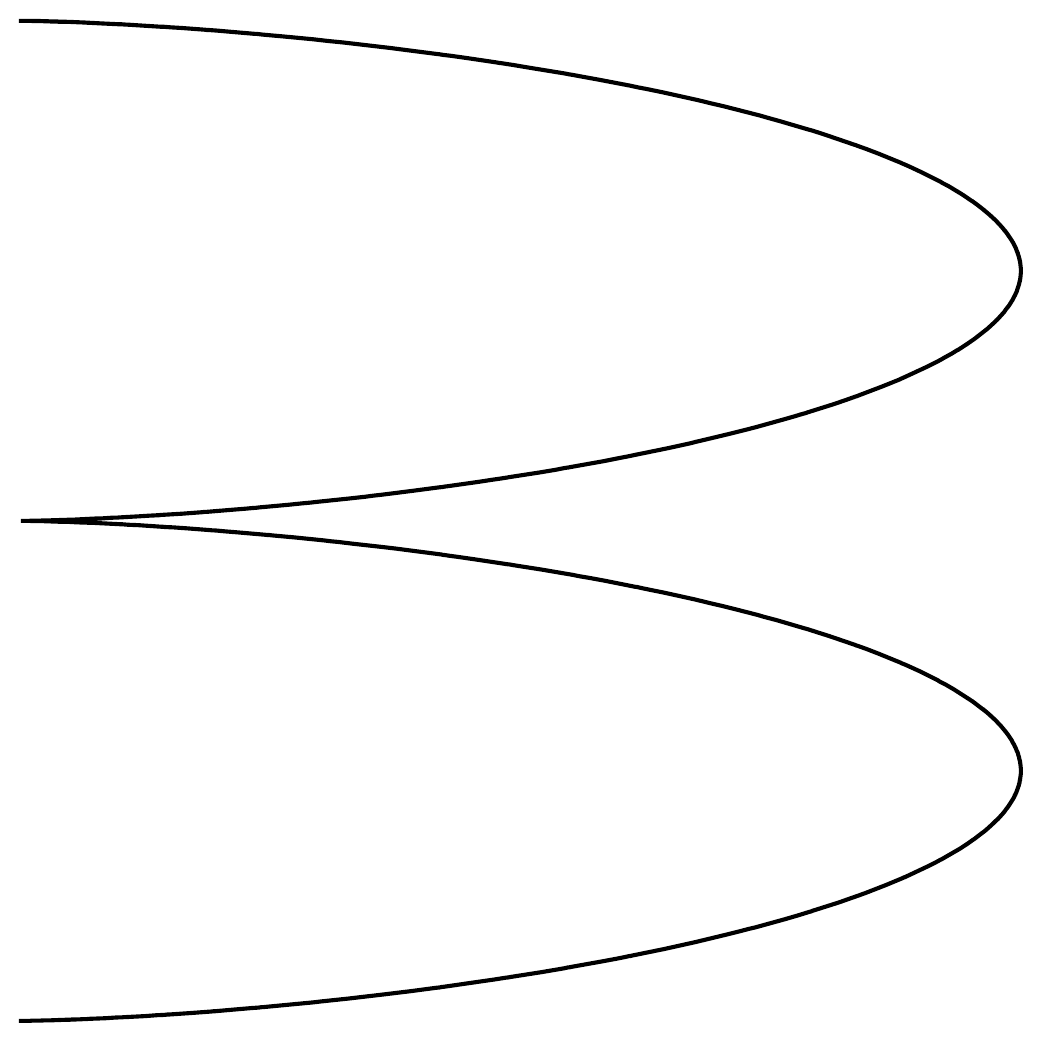}
\hfill
\includegraphics[width=0.32\textwidth]{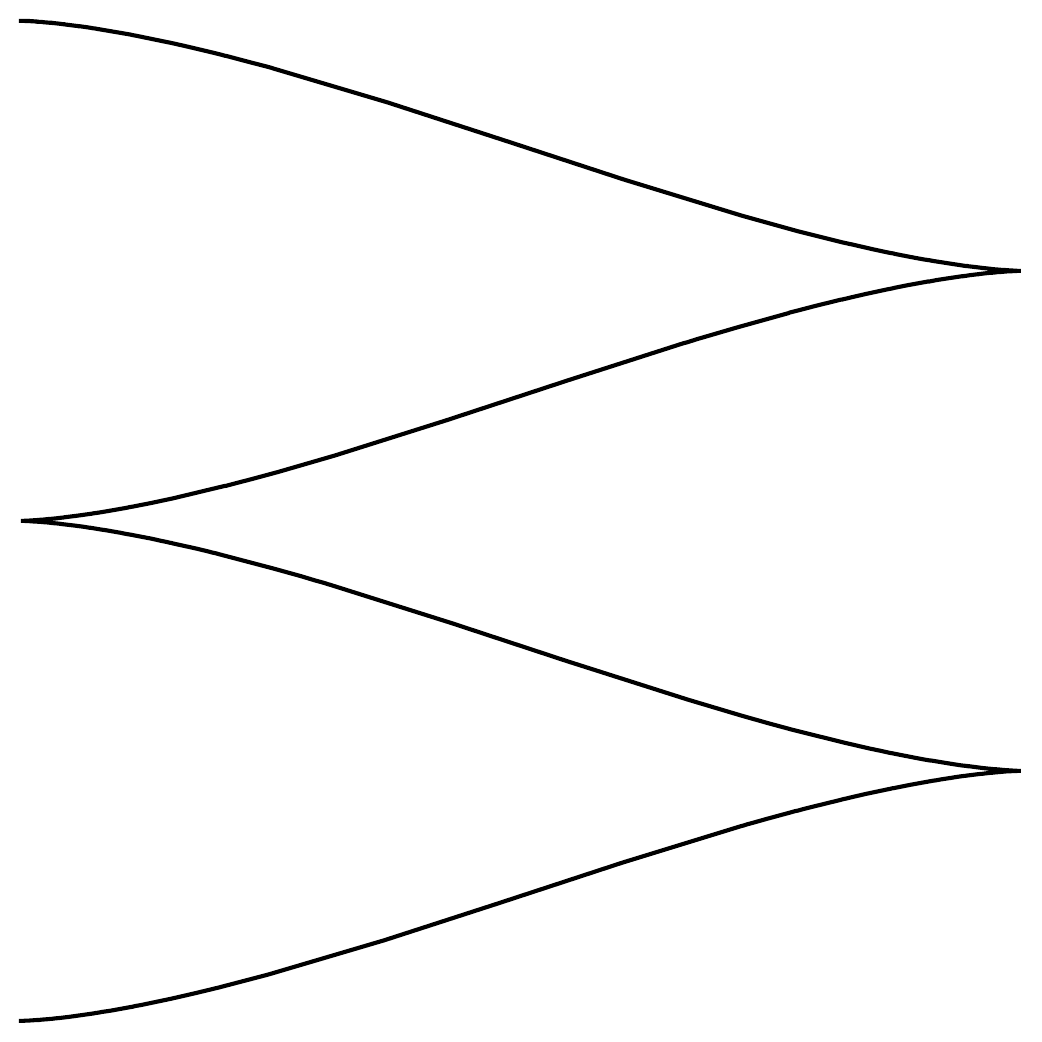}
\hfill
\includegraphics[width=0.32\textwidth]{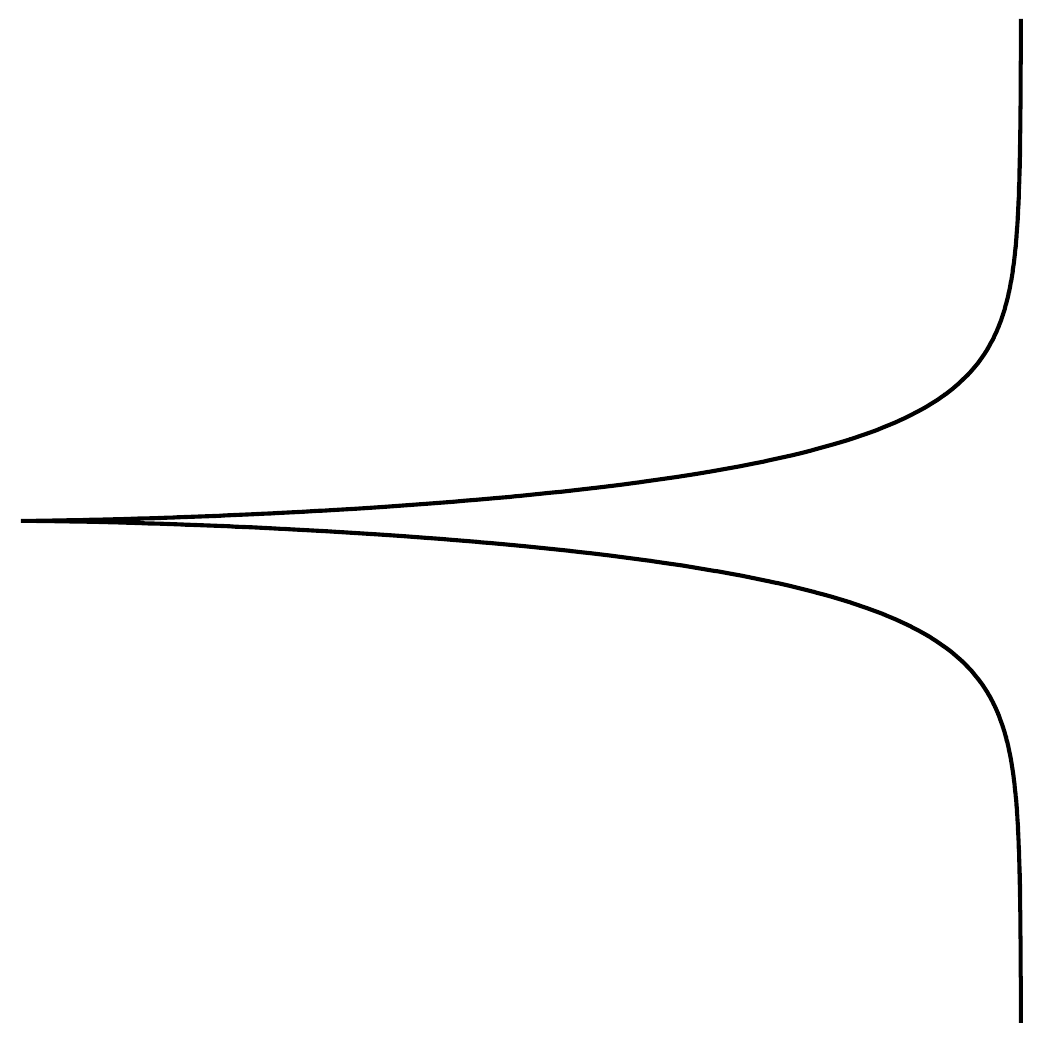}
\\
\parbox[t]{0.32\textwidth}
{\caption{Неинфлексионная геодезическая в $\SE(2)$}\label{fig:xyC1}}
\hfill
\parbox[t]{0.32\textwidth}
{\caption{Инфлексионная геодезическая в $\SE(2)$}\label{fig:xyC2}}
\hfill
\parbox[t]{0.32\textwidth}
{\caption{Критическая геодезическая  в $\SE(2)$}\label{fig:xyC3}}
\end{figure}

Группа $\Isom(\SE(2))$ была вычислена в работе \cite{bicycle}: там показано, что
$$
\Isom(\SE(2)) \cong \SE(2) \rtimes (\Z_2 \times \Z_2).
$$
Группа $\SE(2)$ действует на себе левыми сдвигами. Один сомножитель $\Z_2$ означает отражение в какой-либо оси на плоскости $(x,y)$, а другой сомножитель $\Z_2$ --- разворот.
В частности, все однопараметрические подгруппы в $\Isom(\SE(2))$ суть однопараметрические подгруппы в группе~$\SE(2)$.

\begin{theorem}\label{th:se2_GO}
Однородными геодезическими в группе $\SE(2)$ являются только геодезические $(1)$  и $(2)$, указанные в примере {\em \ref{ex:se2}}. 

Поэтому $\SE(2)$  не является геодезически орбитальным пространством.
\end{theorem}
\begin{proof}
Вычислим однопараметрические подгруппы в $\Isom(\SE(2))$, т.е. в $\SE(2)$. Алгебра Ли группы Ли $\SE(2)$ есть
$$
\g = \se(2) = \spann(E_{13}, E_{23}, E_{21}-E_{12}),
$$
где $E_{ij}$ есть $3 \times 3$  матрица с единственным ненулевым элементом с строке $i$  и столбце $j$, равным~1. Однопараметрическая подгруппа, соответствующая ее элементу $X = a E_{13} + b E_{23} + c (E_{21}-E_{12})$  есть $e^{s X} = (x(s), y(s), \theta(s))$,  где координаты $x, y, \theta$  удовлетворяют задаче Коши
\begin{align*}
&\dot x = - c y + a, & x(0) = 0, \\
&\dot y =  c x + b, & y(0) = 0, \\
&\dot \theta =  c, & \theta(0) = 0.
\end{align*}
Поэтому $\theta(s) = c s$ и
\begin{align*}
&x(s) = as, \quad y(s) = b s \qquad &\text{при } c = 0, \\
&x(s) = \frac bc (\cos c s - 1) + \frac ac \sin cs, \quad y(s) = \frac bc \sin cs + \frac ac (1-\cos c s)  
\qquad &\text{при } c \neq 0.
\end{align*}
Орбита однопараметрической подгруппы  $\{ e^{sX}\} \subset \SE(2)$, проходящая через точку $(x_0, y_0, \theta_0) \in \SE(2)$,  имеет вид:
\begin{align*}
&(x_0 + a s, y_0 + b s, \theta_0) &\text{при } c = 0, \\
&\left(\left(x_0 + \frac bc\right) \cos cs + \left( \frac ac-y_0\right) \sin cs - \frac bc, \left(y_0 - \frac ac\right) \cos cs + \left(\frac bc+x_0\right) \sin cs + \frac ac, \theta_0+ c s\right) &\text{при } c \neq 0.
\end{align*}
Проекция  этой орбиты на плоскость $(x,y)$  есть следующая кривая:
\begin{itemize}
\item[$(1)$]
прямая при $c = 0$, $a^2 + b^2 \neq 0$, 
\item[$(2)$]
точка при $c \neq 0$, $\left(x_0 + \frac bc\right)^2 + \left( \frac ac-y_0\right)^2 = 0$, 
\item[$(3)$]
окружность при $c \left(\left(x_0 + \frac bc\right)^2 + \left( \frac ac-y_0\right)^2\right) \neq 0$. 
\end{itemize}
Из описания проекций геодезических на плоскость $(x, y)$ следует, что
однопараметрические подгруппы в случае (3) не могут быть геодезическими.
Значит, только геодезические из примера~\ref{ex:se2} однородны (случаи (1), (2)).
\end{proof}

Таким образом, на группе $\SE(2)$:
\begin{itemize}
\item
все геодезические эквиоптимальны (следствие \ref{cor:GO}),
\item
однородны только геодезические типов (1), (2) примера \ref{ex:se2} (теорема \ref{th:se2_GO}).
\end{itemize}
Поэтому группа $\SE(2)$  эквиоптимальна, но не геодезически орбитальна. 

Глубинная причина  эквиоптимальности группы $\SE(2)$ остается скрытой.
С другой стороны, группы, перечисленные в п. 1, 2 следствия \ref{cor:GO} геодезически орбитальны и эквиоптимальны \cite{podobryaev_GO}. Отметим также, что стандартная левоинвариантная субриманова структура на свободной двуступенной группе Карно с 3-мя образующими (вектор роста $(3,6)$) \cite{myasnich36} эквиоптимальна \cite{podobryaev_GO}, потому геодезически орбитальна по следствию \ref{cor:GO}.

\bigskip
Автор выражает благодарность А.В. Подобряеву за полезные обсуждения этой работы.

\end{document}